\documentclass[10pt]{amsart}
\usepackage{amsmath}
\usepackage{amsfonts}
\usepackage{amsthm}
\usepackage{amssymb}
\usepackage{mathrsfs}
\usepackage{array}
\usepackage{latexsym}
\usepackage{verbatim}
\usepackage{a4wide}
\usepackage{enumerate}
\usepackage{lipsum}
\usepackage[all]{xy}
\usepackage{tikz-cd}  

\usepackage{hyperref}

\numberwithin{equation}{section}

\renewcommand{\thetheoremName}

\newcommand{\IC}{\mathbb{C}}
\newcommand{\IF}{\mathbb{F}}

\newcommand{\IN}{\mathbb{N}}
\newcommand{\IG}{\mathbb{G}}

\newcommand{\IQ}{\mathbb{Q}}

\newcommand{\IZ}{\mathbb{Z}}

\newcommand{\calF}{\mathcal{F}}
\newcommand{\calG}{\mathcal{G}}

\newcommand{\calO}{\mathcal{O}}

\newcommand{\calS}{\mathcal{S}}

\newcommand{\im}{\mathfrak{m}}

\def\Hom{\mathrm{Hom}}
\def\End{\mathrm{End}}

\def\Gal{\mathrm{Gal}}

\def\Sp{\mathrm{Sp}}

\def\Spf{\mathrm{Spf}}

\DeclareMathOperator\Stab{Stab}
\DeclareMathOperator\Ab{Ab}
\DeclareMathOperator\Nil{Nil}
\DeclareMathOperator\Grp{Grp}
\DeclareMathOperator\Cart{Cart}
\DeclareMathOperator\Fil{Fil}
\DeclareMathOperator\Preper{Preper}

\newtheorem{theorem}{Theorem}[section]
\newtheorem*{theorem*}{Theorem}

\newtheorem{question}[theorem]{Question}

\newtheorem{lemma}[theorem]{Lemma}
\newtheorem{prop}[theorem]{Proposition}
\newtheorem{corollary}[theorem]{Corollary}
\newtheorem{conj}{Conjecture}

\theoremstyle{definition}

\theoremstyle{remark}

\newtheorem{remark}[theorem]{Remark}

\begin{document}
\title{On $p$-adic versions of the Manin--Mumford Conjecture}
\author{Vlad Serban}

\thanks{This research was partially supported by the Fields Institute for Research in Mathematical
Sciences (\href{www.fields.utoronto.ca}{www.fields.utoronto.ca}) and START-prize Y-966 of the Austrian Science Fund (FWF) under P.I. Harald Grobner}

\address{Vlad Serban, Department of Mathematics, EPFL, 1015 Lausanne, Switzerland}
\email{vlad.serban@epfl.ch}

\bibliographystyle{alpha}

\begin{abstract}
We prove $p$-adic versions of a classical result in arithmetic geometry stating that an irreducible subvariety of an abelian variety with dense torsion has to be the translate of a subgroup by a torsion point. We do so in the context of certain rigid analytic spaces and formal groups over a $p$-adic field $K$ or its ring of integers $R$, respectively. In particular, we show that the rigidity results for algebraic functions underlying the so-called Manin-Mumford Conjecture generalize to suitable $p$-adic analytic functions. In the formal setting, this approach leads us to uncover purely $p$-adic Manin-Mumford type results for formal groups not coming from abelian schemes. Moreover, we observe that a version of the Tate-Voloch Conjecture holds in the $p$-adic setting: torsion points either lie squarely on a subscheme or are uniformly bounded away from it in the $p$-adic distance. 
\end{abstract}

\maketitle

\section{Introduction} 
The Manin--Mumford Conjecture concerns the intersection of the torsion points on a semi-abelian variety $G$ with a subvariety and characterizes the situations where said intersection is exceptionally large. More precisely, it states that the torsion points are dense on an irreducible subvariety $V\subset G$ if and only if $V$ is a \emph{special} subvariety of $G$: the translate by a torsion point of an algebraic subgroup. In its classical form, when $G$ is an abelian variety, this was established by M. Raynaud \cite{MR688265,MR717600}, whereas the case when $G$ is an algebraic torus already appears in work of S. Lang \cite{MR0190146}. There have been many subsequent proofs and generalizations; we refer the reader to surveys such as U. Zannier's book \cite{MR2918151} for more on this active research topic of so-called \emph{unlikely intersections}. \par

In this article, we examine how in a $p$-adic setting some of these statements may be strengthened. More precisely, assume for simplicity $G=\IG_m^n$ is an algebraic torus and $V$ an embedded irreducible affine curve. The conjecture then amounts to showing that if polynomials in the affine coordinate ring of $V$ vanish at infinitely many $n$-tuples of roots of unity $(\zeta_1,\ldots,\zeta_n)\in\IG_m^{n,tor}$, then $V$ must be the translate of a subtorus by a torsion point.\par
Let now $R$ be a complete, discretely valued subring of $\calO_{\IC_p}$ for some prime $p$ and $K:=R[1/p]$. One may then ask if a similar statement holds for power series $f\in K[[X_1,\ldots,X_n]]$ vanishing at infinitely many $n$-tuples of roots of unity, however this runs into obvious counterexamples and convergence issues. We present two ways to remedy this situation.\par
Firstly, one may consider formal power series with integral coefficients $f\in R[[X_1,\ldots,X_n]]$ and restrict to torsion points lying on the interior of the $p$-adic unit disk centered at the identity, which have $p$-power order. This fixes convergence issues and in this framework the author showed in \cite{Serban:2016aa} that an irreducible closed formal subscheme of the formal torus $\widehat{\IG}_m^n$ containing a dense set of torsion points indeed has to be the translate of a formal subtorus by a torsion point. A similar result holds more generally after replacing $\widehat{\IG}_m$ by a one-dimensional Lubin--Tate formal group (\cite[Theorem 3.7.]{Serban:2016aa}). \par
More generally, one may consider an $n$-dimensional formal Lie group $\calF$ over $R$ and ask if a similar statement holds. Let $d(X,-)$ denote a suitable $p$-adic distance to a (formal) scheme $X$. We prove the following result, which may be viewed as a $p$-adic formal version of the classical Manin--Mumford Conjecture, since an $n$-dimensional abelian variety over $R$ gives rise to an $n$-dimensional $p$-divisible formal Lie group by completion along the identity section: 
\begin{theorem}\label{maintheoremformal}
Let $\calF:=\Spf(R[[X_1,\ldots, X_n]])$ denote an $n$-dimensional $p$-divisible formal Lie group over $R$ and let $X\hookrightarrow \calF$ be a closed formal subscheme over $R$. Then exactly one of the following occurs:
\begin{enumerate}
\item There exists $\varepsilon>0$ such that there are only finitely many torsion points $Q\in\calF[p^\infty]$ with $d(X,Q)<\varepsilon$.
\item There are infinitely many torsion points of $\calF$ on $X(\IC_p)$ and $X$ contains (possibly after passing to a finite extension of $R$) the translate by a torsion point of a positive dimensional formal subgroup. 
\end{enumerate}
\end{theorem}
 The $p$-primary torsion $\calF[p^\infty]$ exhausts all of the finite order points on the formal group. This general case required a new proof, since one does not have a large ring of endomorphisms at one's disposal as in the Lubin-Tate case considered in \cite{Serban:2016aa}. Rather, we adapt a proof strategy from the classical algebraic situation due to Boxall \cite{MR1345173}. We note that a slightly stronger statement appears plausible, see Question \ref{strongconjectureformal} and the ensuing discussion. \par

Furthermore, we observe genuinely new phenomena in this $p$-adic formal setting: firstly, Theorem \ref{maintheoremformal} exhibits unlikely intersection results for formal groups which do not come from abelian varieties.
Secondly, the dichotomy above is a version of the Tate--Voloch Conjecture (see, e.g., \cite{MR1405976,TV1Scanlon,TV2Scanlon}) in this formal setting. It states that over a $p$-adic base torsion points should either lie on a subvariety of a semi-abelian variety or be uniformly bounded away from it given a $p$-adic distance function. See also the work of A. Neira \cite{MR2705372} establishing this for functions in the Tate algebra $T_n(K):=\{f=\sum_Ia_IX^I\in K[[X_1,\ldots, X_n]]\vert a_I\to 0 \}$ which vanish at torsion points of the multiplicative group.\par
Along those lines, we wish to point out a second way of making sense of a $p$-adic strengthening. Returning to the case when the ambient group is a torus $\IG_m^n$, we may instead consider all of the torsion points, but require the analytic functions to also converge on the boundary of the unit disk (by this we mean points in $\IC_p^n$ with trivial $p$-adic valuation), leading us to consider functions in $T_n(K)$ only. We then obtain in Section \ref{sec:rigid} a rigid analytic strengthening of Lang's multiplicative Manin--Mumford result: 
\begin{theorem}
Let $X=\Sp(A)$ denote the rigid space associated to an affinoid $K$-algebra $A:=T_n(K)/I$ for some finite extension $K/\IQ_p$. If $X(\IC_p)$ contains infinitely many torsion points of $\IG_m^n(\IC_p)$, then $X$ contains a positive-dimensional algebraic torus. Moreover, the same conclusion holds if infinitely many torsion points approach $X(\IC_p)$ for a suitable $p$-adic distance. 
\end{theorem}
It is presently not clear to the author how such results can be extended or formulated beyond the multiplicative case.\par
Finally, we would like to draw the reader's attention to the connection of our results with the theory of non-archimedean dynamical systems, as laid out in J. Lubin's paper \cite{MR1310863}. Letting $\calF$ denote a one-dimensional formal group over $R$, the preperiodic points of the multiplication-by-$n$ map $[n](X)\in X\cdot R[[X]]$ for an integer $n\geq 2$ are precisely the torsion points of the formal group $\calF$. Viewing therefore the torsion points as preperiodic points of a $p$-adic dynamical system attached to $\calF$, our results imply that if the preperiodic points of two such dynamical systems have infinite intersection, they must come from the same dynamical system, as proved by L. Berger \cite[Theorem A]{2018arXiv181105824B} via a different approach. Moreover, we obtain a slightly more general result, and again it suffices for infinitely many preperiodic points to be $p$-adically close to arrive at the same conclusion. Concretely, we have the following corollary to Theorem \ref{maintheoremformal}:
\begin{corollary}\label{cor:dynamic}
Let $\calF$ and $\calG$ be two finite height one-dimensional formal groups over $R$. Let $h\in X\cdot R[[X]]$ denote a power series satisfying $h'(0)\neq 0$. If for all $\varepsilon>0$ there are infinitely many pairs of preperiodic points $(\zeta,\xi)\in \Preper(\calF)\times \Preper(\calG)$ such that
$$\vert h(\zeta)-\xi\vert_p<\varepsilon,$$
then $h\in\Hom(\calF,\calG)$.
\end{corollary}
This answers the questions raised in \cite[Section 4.3]{2018arXiv181105824B}. The fact that unlikely intersection results have dynamical analogues and the connections between these statements is well-known in the classical setting, and dynamical unlikely intersection results have been a fruitful area of research. We refer the reader to \cite[Theorem 1.2.]{MR2817647} for an unlikely intersection result for preperiodic points as in Corollary \ref{cor:dynamic} in the setting rational functions of degree at least $2$ over $\IC$. See also for instance recent work \cite{MR3126567,MR3468757,MR3801489} and the dynamical conjectures laid out in \cite[Section 4]{MR2408228}. It would be interesting to invesitgate how to extend the results for $p$-adic dynamical systems further. For possible avenues, we refer the reader to the discussion in \cite[Section 4]{2018arXiv181105824B}. In particular, do there have to be formal groups lurking in the background for unlikely intersection results for $p$-adic dynamical systems to hold?

\section{Manin--Mumford for formal groups}\label{sec:formal}
Let $R$ denote a complete discretely valued ring of mixed characteristic $(0,p)$ and let $\calF:=\Spf(R[[X_1,\ldots, X_n]])$ denote a smooth, $n$-dimensional formal group over $R$. These can be defined via a formal group law or as functors on the category of nilpotent $R$-algebras $\calF:\Nil_R\to \Ab$ that are formally smooth, left exact and commute with arbitrary direct sums. We moreover require $p$-divisibility of $\calF$, meaning that multiplication $[p]:\calF\to\calF$ is an \emph{isogeny}, i.e. has a representable kernel. We slightly abuse notations and write $\calF(A)$ for the points of $\calF$ over the (topologically) nilpotent elements of an $R$-algebra $A$; for instance $\calF(\IC_p)$ denotes the points on the $p$-adic unit ball $\im_{\IC_p}^n$ of the formal group. We denote the torsion points of $\calF$ on that ball, which all have $p$-power order, by $\calF[p^\infty]$.\par
For a closed affine formal subscheme $X\hookrightarrow \calF$, given by $X=\Spf(R[[X_1,\ldots, X_n]]/I)$ for some ideal of definition $I$, our goal is to examine under which circumstances more torsion points than expected lie $p$-adically close to $X$. To that end, for any $P\in\calF(\IC_p)$ we define the distance function
$$d(P,X):=\max_{f\in I}(\vert f(P)\vert_p)$$
given a $p$-adic absolute value $\vert-\vert_p$. Since $R[[X_1,\ldots, X_n]]$ is a Noetherian ring, it suffices to compute this distance on the finitely many generators of $I$. The distance depends on a variety of choices, including the choice of normalization of the $p$-adic absolute value and the choice of coordinates on the formal Lie group. However, the distance from a torsion point to a formal subscheme is invariant under automorphisms of the ambient formal Lie group, and this suffices for our purposes. We also abbreviate $X(\varepsilon):=\{\zeta\in \calF[p^\infty]\vert d(\zeta,X)<\varepsilon \}$.\par
We recall that the classical Manin--Mumford Conjecture can be formulated as follows: for a semi-abelian variety $G$, one defines the \emph{special subvarieties} to be the translates of semi-abelian subvarieties of $G$ by torsion points. The \emph{special points} are just the zero-dimensional special subvarieties. The conjecture then states:
\begin{conj}[Manin--Mumford]\label{conj:mm}
An irreducible component of the Zariski-closure of a set of special points is a special subvariety. 
\end{conj}
Motivated by this result, in our setting a natural candidate for a \emph{special formal subscheme} is the translate by a torsion point in $\calF[p^\infty]$ of a formal subgroup of $\calF$, and the points in $\calF[p^\infty]$ should precisely constitute the \emph{special points}. Combined with the aforementioned expectation (see \cite{MR1405976,MR2185637}) that in this $p$-adic metric there should exist a uniform lower bound on the distance of a special point $\zeta\in\calF[p^\infty]$ to $X$ provided $\zeta\not\in X$, we expect that for $\varepsilon>0$ small enough $X(\varepsilon)$ is finite unless $X$ contains a positive-dimensional special formal subscheme. This implies a slightly weaker result than the analogue of Conjecture \ref{conj:mm}, which we discuss in the next section. \par

The proof strategy we employ is an adaptation of the following argument from the classical algebraic variety setting: a (semi-)abelian variety $G$ acts on a subvariety $V\subset G$ by translation and one therefore defines a stabilizer of $V$, which is an algebraic subgroup passing through the origin (assuming one has arranged for $V$ to pass through the origin). It follows from the classification of commutative group varieties that the connected component of the stabilizer passing through the origin is again a (semi-)abelian variety embedded in $G\cap V$. Hence, after taking quotients, we may as well assume that the stabilizer is finite. In this case, one then produces, using Galois action, a point $Q$ outside the stabilizer such that the translate $T_QV$ of $V$ by $-Q$ also contains infinite torsion. This reduces the proof to the lower-dimensional $T_QV\cap V$, and one inducts.  
This method, due to J. Boxall, provides a proof of Manin--Mumford for $n$-primary torsion, see \cite{MR1345173}. We are able to carry through most of this strategy in our setting, while one has to circumvent some difficulties when dealing with schemes which are not of finite type. \par

In our $p$-adic formal setting, consider the orbits or translates of $X\subset \calF=\Spf(R[[X_1,\ldots, X_n]])$ under the action of $\calF$ as well as the stabilizer of $X$ under that action. For any $Q\in \calF(A)$ we define the translate $T_QX$ on the level of points by $T_QX(A)=\{x-Q\vert x\in X(A)\}$.
\begin{lemma}\label{lemma:translate}
The translate $T_QX$ is represented by a formal scheme and comes with a closed immersion $i_Q:T_QX\hookrightarrow \calF$ over the ring $R[Q]$, which is Noetherian for $Q\in\overline{K}$. 
\end{lemma}
\begin{proof}
It suffices to show this for $X=\Spf(R[[X_1,\ldots, X_n]]/I)$ for some ideal $I$ of the Noetherian ring $R[[X_1,\ldots, X_n]]$.
Let $(\phi_1(X), \ldots, \phi_d(X))$ be generators of the ideal of definition and let $F(X,Y)\in R[[X_1,\ldots, X_n, Y_1,\ldots, Y_n]]$ denote the formal group law. Then the ideal of definition of $T_QX$ is generated by the power series $(\phi_1(F(X,Q)),\ldots, \phi_n(F(X,Q)))$. Moreover, when the extension degree $[K(Q):K]$ is finite $R[Q]$ is Noetherian, as claimed.

\end{proof}
 We also define the stabilizer of the formal subscheme on the level of points via 
$$\Stab(X)(A):=\{P\in\calF(A)\vert X(A)+_\calF P= X(A)\}$$
for a commutative nilpotent $R$-algebra $A$ and proceed to list some properties of the stabilizer. 
\begin{lemma}\label{stabilizerisagroup}
For any commutative nilpotent $R$-algebra $A$, $\Stab(X)(A)$ is a subgroup of $\calF(A)$ and 
$$\Stab(X)(A)=\bigcap_{Q\in X(A)} T_QX(A)$$
as sets. Moreover, if $X$ passes through the identity, then $\Stab(X)(A)$ is contained in $X(A)$. 
\end{lemma}
\begin{proof}
The proof of the claim is straightforward: let $g\in \Stab(X)(A)$, then for all $Q\in X(A)$, we have $g+Q=P\Leftrightarrow g=P-Q$ for some $P\in X(A)$, with addition in the formal group. This shows $g\in T_QX(A)$. Conversely, for $g\in\bigcap_{Q\in X(A)} T_QX(A)$ we may for any $P\in X(A)$ write $g+P=(g_P-P)+P=g_P$ for some unique $g_P\in X(A)$, showing the reverse inclusion. It also follows immediately from the definition that $\Stab(X)(A)$ is stable under addition in the formal group and contains the identity. Moreover, if $g\in \Stab(X)(A)$, then $X(A)=-g+g+X(A)=-g+X(A)$, so $\Stab(X)$ is stable under inverses and $\Stab(X)(A)$ is a subgroup of $\calF(A)$. Finally, if the identity $e$ is in $X$, clearly $\Stab(X)(A)\subset T_eX(A)=X(A)$.

\end{proof}

\begin{lemma}\label{lemma:infinitestab}
Assume that $\Stab(X)(\overline{K})$ is infinite and $X$ passes through the identity. Then the stabilizer is a formal subgroup scheme of $\calF$ over the valuation ring $\calO_L$ of some finite extension $L$ of $K$. It is contained in $X$ and has a positive-dimensional connected component.
\end{lemma}
\begin{proof}
We first claim that there exist finitely many algebraic points $Q_1,\ldots, Q_k \in X(\overline{K})$ such that
$$\Stab(X)(\overline{K})=T_{Q_1}X(\overline{K})\cap\cdots\cap T_{Q_k}X(\overline{K}),$$
and we may as well take $Q_1$ to be the identity. Assume the claim does not hold. Then for any component $Y$ of a finite intersection $T_{Q_1}X\cap\cdots\cap T_{Q_k}X$ we may then find $Q_{k+1}$ such that all components of $Y\cap T_{Q_{k+1}}X$ have dimension $d\leq \max(0,\dim Y-1)$. Since $\Stab(X)(\overline{K})=\bigcap_{Q \in X(\overline{K})}T_QX(\overline{K})$ as sets, proceeding in this way we may after a finite number of steps arrange that $\Stab(X)(\overline{K})\subseteq Z(\overline{K})$ for some zero-dimensional formal scheme $Z$. It would follow that $\Stab(X)(\overline{K})$ is finite, a contradiction. \\
Now set $\calS:=T_{Q_1}X\cap\cdots\cap T_{Q_k}X$ with $Q_1,\ldots, Q_k \in X(\overline{K})$ as in the claim, so that $\calS$ is by Lemma \ref{lemma:translate} a closed formal subscheme of $\calF$ over $\calO_L$ for some finite extension $L$ of $K$. It follows from the claim that $\calS(\overline{K})$ is an infinite subgroup of $\calF(\overline{K})$ contained in $X$. A fortiori $\calS(A)$ is a subgroup of $\calF(A)$ for any nilpotent commutative $R$-algebra $A$: denoting by $\phi\in R[[X_1,\ldots,X_n,Y_1,\ldots, Y_n]]$ the formal group law of $\calF$ and writing $\calS=\Spf(\calO_L[[X_1,\ldots,X_n]]/J)$, for any $f\in J$ we have that $f\circ\phi$ vanishes along the dense set $\calS(\overline{K})\times \calS(\overline{K})$ of $\calS\times \calS$ so that $\calS$ is stable under the formal group law. Similarly we deduce stability under taking inverses so that $\calS$ is the desired formal subgroup scheme. The remaining statements are clear. 
\end{proof}

It remains to examine the case when $\Stab(X)(\overline{K})$ is finite, which we tackle using the Galois action on the points of our formal schemes. 
Write $\Gal(L)$ for the absolute Galois group of a $p$-adic field $L$. This group acts on the $\IC_p$-points of the formal group. We utilize a $p$-divisible formal Lie groups analogue of \cite[Lemma 2]{MR1426536} concerning the action of inertia which still holds in this setting, mutatis mutandis:


\begin{lemma}\label{Boxallformal}
Let $\calF$ be a (finite dimensional) $p$-divisible formal Lie group defined over $R$ and write $K=R[1/p]$ and $L=K(\calF[p^r])$ for a fixed $r>0$ with $r\neq 1$ if $p=2$. Suppose $P$ is a torsion point in $\calF(\IC_p)$ not defined over $L$, then there exists $s\in \Gal(L)$ such that $s(P)-P\in \calF[p^r]\setminus \calF[p^{r-1}]$, where subtraction is given by the formal group law of $\calF$. 
\end{lemma}
\begin{proof}
Let $n$ be minimal with the property that $[p^n]P\in L$; this is always possible by our assumptions. Also let $P_i:=[p^{n-r-i+1}]P$ and pick $s_1\in \Gal(L)$ such that $s_1([p^{n-1}]P)\neq[p^{n-1}]P$. Finally set $s_i=s_1^{p^{i-1}}$ and $Q_i:=s_i(P_i)-(P_i)$, noting that addition and subtraction are using the formal group law throughout the proof. We prove by induction on $i$ that 
$$Q_i\in \calF[p^r]\setminus \calF[p^{r-1}].$$
It then follows that $s_{n-r+1}$ proves the Lemma, since $P_{n-r+1}=P$. The statement holds for $i=1$, as $[p^r]Q_1=s_1([p^n]P)-[p^n]P=0$ and $[p^{r-1}]Q_1=s_1([p^{n-1}]P)-[p^{n-1}]P\neq 0$  by construction.
Let now $i>1$ and assume the claim holds for $Q_i$. First observe that $s_i^j(P_i)-P_i=[j]Q_i$ for $j\geq1$: this is true for $j=1$ and follows easily by induction since 
$$s_i^{j+1}(P_i)-P_i=s_i([j](Q_i))+s_i(P_i)-P_i=[j+1]Q_i, $$
using that $Q_i\in \calF[p^r]$ and is therefore fixed by $s_i$ by induction hypothesis.
Taking $j=p$ yields $[p]Q_{i+1}=[p](s_{i+1}(P_{i+1})-P_{i+1})=s_i^p([p]P_{i+1})-[p]P_{i+1}=s_i^p(P_i)-P_i=[p]Q_i\in \calF[p^{r-1}]$ and therefore $Q_{i+1}\in \calF[p^r]$. Now set $R_i:=s_i(P_{i+1})-P_{i+1}$ and $S_i:=s_i(R_i)-R_i$ so that $[p]R_i=Q_i$ and $[p]S_i=s_i(Q_i)-Q_i=0$. By induction on $j$ one obtains $s_i^j(R_i)=R_i+[j]S_i$. Summing for $p$ odd one has (all sums using the formal group law):
$$\sum_{j=0}^{p-1}s_i^j(R_i)=[p]R_i+[p(p-1)2^{-1}]S_i=[p]R_i=Q_i$$
and also 
$$\sum_{j=0}^{p-1}s_i^j(R_i)=\sum_{j=0}^{p-1}s_i^j(s_i(P_{i+1})-P_{i+1})=s_i^p(P_{i+1})-P_{i+1}=Q_{i+1}$$
and therefore $Q_i=Q_{i+1}$, proving the claim. Similarly if $p=2$. 

\end{proof}

\begin{lemma}\label{translateclose}
Let $\calF$ be a formal group over $R$ and let $X$ be a formal subscheme with $d(\zeta, X)\leq \varepsilon$ for some torsion point $\zeta\in \calF[p^\infty]$. Then for any $\sigma\in \Gal(K)$ we have
$$d(\zeta, T_{\sigma(\zeta)-\zeta}X)\leq \varepsilon.$$
\end{lemma}
\begin{proof}
Writing $\zeta=\sigma(\zeta)+(\zeta-\sigma(\zeta))$, we see that it suffices to show that $d(\sigma(\zeta), X)\leq \varepsilon$. But for any $\phi$ in the ideal of definition of $X$, we have that $\phi(\sigma(\zeta))=\sigma(\phi(\zeta))$ and by assumption $\vert\phi(\zeta)\vert_p\leq \varepsilon$. Since $\sigma$ acts by continuous automorphisms it follows that $\vert\sigma(\phi(\zeta))\vert_p\leq \varepsilon$, as desired.
\end{proof}
We may now prove the main theorem.

\begin{proof}[Proof of Theorem \ref{maintheoremformal}]

We proceed by induction on $\dim(X)$ since all of the rings considered are Noetherian. Without loss of generality assume $X$ is irreducible and passes through the identity of $\calF$. We also assume $\Stab(X)(\overline{K})\cap \calF[p^\infty]$ is finite since otherwise the result holds by Lemma \ref{lemma:infinitestab}. Fix $r$ large enough so that $\Stab(X)(\IC_p)\cap\calF[p^\infty]\subseteq \calF[p^{r-1}]$.
By Lemma \ref{Boxallformal}, for any $\zeta\not\in \calF[p^{r}]$ there exists $\sigma \in \Gal(F)$ such that $\sigma(\zeta)-\zeta\in \calF[p^{r}]\setminus \calF[p^{r-1}]$, with subtraction in the formal group. For any $\varepsilon>0$, recall we denote $X(\varepsilon)=\{\zeta\in \calF[p^\infty]\vert d(\zeta,X)\leq \varepsilon\}$. Then by Lemma \ref{translateclose}, we may for all $\varepsilon>0$ write 
$$X(\varepsilon)\subseteq \calF[p^{r-1}]\cup\bigcup_{Q\in \calF[p^{r}]\setminus \calF[p^{r-1}]} (X\cap T_{Q}X)(\varepsilon).$$
Since this is a finite union it suffices to prove the theorem for each $X\cap T_{Q}X$. Since $Q\not\in\Stab(X)$ the dimension of each irreducible component of $X\cap T_{Q}X$ is strictly less than $\dim(X)$. Moreover, it follows from Lemma \ref{lemma:translate} that each of the $X\cap T_{Q}X$ is a closed subscheme over some Noetherian local base. We may therefore proceed by induction and reduce to the case when $X$ is a point, where the result holds since the torsion points on $\calF(\IC_p)$ are discrete in the $p$-adic topology. 
\end{proof}
We conclude this section by deducing the application to nonarchimedean dynamical systems in the sense of Lubin \cite{MR1310863}:

\begin{corollary}
Let $\calF$ and $\calG$ be two finite height one-dimensional formal groups over $R$. Let $h\in Y\cdot R[[Y]]$ denote a power series satisfying $h'(0)\neq 0$. If for all $\varepsilon>0$ there are infinitely many pairs $(\zeta,\xi)\in \calF[p^\infty]\times \calG[p^\infty]$ such that $\vert h(\zeta)-\xi\vert_p<\varepsilon$, then $h\in\Hom(\calF,\calG)$.
\end{corollary}
\begin{proof}
Assume that for all $\varepsilon>0$ we have $\vert h(\zeta)-\xi\vert_p<\varepsilon$ for infinitely many pairs. Consider the formal subscheme $X=\Spf(R[[Y_1,Y_2]]/(h(Y_1)-Y_2)$ of the two-dimensional $p$-divisible formal Lie group $\calF\times\calG$. Then by Theorem \ref{maintheoremformal}, $X$ contains a positive dimensional formal subgroup and since $X$ is irreducible $X$ is a formal subgroup. It now follows that $h\in\Hom(\calF,\calG)$. For example, let $f(Y_1):=[2]_\calF(Y_1)$ and $g(Y_2):=[2]_\calG(Y_2)$ denote the duplication maps on $\calF$ and $\calG$ respectively. Since $X$ is stable under multiplication by $2$, it follows that $h\circ f= g\circ h$. This commutation relation implies 
$h\in\Hom(\calF,\calG)$ by Lemma 4.4 of \cite{2018arXiv181105824B}. 
\end{proof}

\section{Some Remarks}\label{section:remarks}
Motivated by the classical Manin--Mumford and Tate--Voloch Conjectures and by the cases when $\calF$ is a product of one-dimensional formal groups studied in \cite{Serban:2016aa}, it seems reasonable to expect that a stronger result holds and that, for $\varepsilon>0$ small enough, one can exhaust all torsion in $X(\varepsilon)$ by a finite union of translates of formal subgroups contained in $X$ by torsion points. We record this here:
\begin{question}\label{strongconjectureformal}
Determine for which $p$-divisible formal Lie groups $\calF:=\Spf(R[[X_1,\ldots, X_n]])$ over $R$ one has that for any closed affine formal subscheme $X\hookrightarrow \calF$ over $R$ exactly one of the following two situations occurs: 
\begin{enumerate}
\item There exists $\varepsilon>0$ such that there are only finitely many torsion points $Q\in\calF[p^\infty]$ with $d(X,Q)<\varepsilon$.
\item There are infinitely many torsion points of $\calF$ on $X(\IC_p)$. Moreover, there exists a finite union of positive dimensional formal subgroups $\bigcup_{i=1}^NH_i$ over a finite extension of $R$ and torsion points $\xi_i\in \calF[p^\infty]$ such that the translates $T_{\xi_i}H_i$ are contained in $X$ and the set 
$$S_\varepsilon:=\{Q\in \calF[p^\infty]\setminus \left(\cup_{i=1}^NT_{\xi_i}H_i(\IC_p)\cap\calF[p^\infty]\right)\vert d(X,Q)< \varepsilon\}$$
is finite for $\varepsilon$ small enough.
\end{enumerate}
\end{question}
For $\calF=\widehat{\IG}^n_m$ or a product of one-dimensional Lubin--Tate formal groups, this result is proved in 
\cite{Serban:2016aa}. Such a result could be established more generally via our strategy if one were able to reduce to the case of finite stabilizers by passing to suitable quotients by formal subgroups: in the classical setting where the ambient group is an abelian variety $G$, if the stabilizer is infinite it follows that the connected component passing through the identity $\Stab^0(X)$ is an abelian subvariety of $G$. One can then form the quotients $X/\Stab^0(X)\to G/\Stab^0(X)$ and in this way reduce to the case of finite stabilizer. However, we have not succeeded so far in carrying through this strategy for arbitrary $p$-divisible formal Lie groups. For instance, in general ther are obstacles to taking suitable quotients in the category of $p$-divisible formal Lie groups. \\
To illustrate this point, recall that formal Lie groups over $R$ can be studied in terms of various semi-linear algebra constructions. We mention Cartier theory and derive a couple consequences thereof here, giving T. Zink's book \cite{MR767090} as a reference. 
\begin{remark}\label{remark:pdiv}
Alternatively, one could view formal groups over $R$ as connected $p$-divisible groups and the latter objects are well-studied. To name a few results, building on work of J-M. Fontaine and others, C. Breuil showed (for odd $p$, see \cite[Th\' eor\`eme 1.4.]{MR1804530}) that $p$-divisble groups over the ring $R$ of integers of a discretely valued complete extension $K$ of $\IQ_p$ with perfect residue field are equivalent to the category of $\Gal(K)$-stable $\mathbb{Z}_p$-lattices which are crystalline representations of $\Gal(K)$ with Hodge-Tate weights in $\{0,1\}$. Over $\calO_{{\IC}_p}$, we also have a classification for $p$-divisible groups due to P. Scholze and J. Weinstein \cite[Thm B]{ScholzeWeinstein}. 
\end{remark}

Let $k$ denote a commutative ring with unit. Viewing formal groups as functors $G:\Nil_k\to \Ab$, the first main theorem of Cartier theory gives a Yoneda-type isomorphism 
$$\Hom(\Lambda,G)\cong G(k[[X]]),$$
where the formal group $\Lambda$ is defined by $\Lambda(N)=1+tk[t]\otimes_k N$ for nilpotent $k$-algebras $N$. The group $G(k[[X]])$ is called the group of curves on $G$. Setting $\Cart(k):=\End(\Lambda)^{op}$, it follows that the ring $\Cart(k)$ is isomorphic to the curves on $\Lambda$ and that a formal group gives rise to a module over the Cartier ring $\Cart(k)$. \par
The second main theorem of Cartier theory gives a characterization of modules arising in this way.
We state its local version when $k$ is a $\IZ_{(p)}$-algebra : conjugating $\Cart(k)$ with a suitable idempotent, we obtain the local Cartier ring $\Cart_p(k)$. It can be identified with the endomorphism ring of the formal Witt vectors and can be described explicitly in terms of topological generators $F,V$ and $[a]$ for $a\in k$. These satisfy $FV=p$ and a number of commutation relations (see \cite[4.17.]{MR767090}). 
The topology coming from the filtration $$\Fil^n(\Cart_p(k)):=V^n\Cart_p(k)$$
turns $\Cart_p(k)$ into a complete, separated topological ring.  As before a formal group gives rise to a (left) $\Cart_p(k)$-module. To characterize modules arising in this way, call a left $\Cart_p(k)$-module $M$ \emph{$V$-flat} if $M/VM$ is a flat $k$-module and \emph{$V$-reduced} if $V:M\to M$ is injective and $M$ is $V$-adically complete. We then have the local version of the second main theorem: 
\begin{theorem}[Thm. 4.23 in \cite{MR767090}]
There is an equivalence of categories:
$$\{\text{smooth formal groups over }k\}\leftrightarrow\{V\text{-flat and }V\text{-reduced left }\Cart_p(k)\text{-modules}\}$$
given by sending a formal group $G$ to the $\Cart_p(k)$-module $M_G$ of $p$-typical curves in $G(k[[X]])$. Moreover, the Lie algebra of the formal Lie group is isomorphic to $M_G/VM_G$. \end{theorem}

Concerning quotients of formal groups, consider therefore $M_1\to M_2$ a homomorphism of $V$-flat, $V$-reduced $\Cart_p(k)$-modules and suppose the induced map $M_1/VM_1\to M_2/VM_2$ is injective. 
Then so is $M_1\to M_2$, and letting $M$ denote the quotient module, by the Snake Lemma $V:M\to M$ is an injective left $\Cart_p(k)$-module homomorphism and for $n\geq 1$ we have short exact sequences 
$$0\to M_1/V^nM_1\to M_2/V^nM_2\to M/V^nM\to 0.$$
This shows that the quotient is $V$-reduced. However, taking $n=1$ above shows that the quotient module $M$ is not necessarily $V$-flat. Thus there seems to be no way of taking quotients of smooth formal groups in this generality. \par
Using Cartier modules, it is also not hard to show that lifting formal Lie groups from a characteristic $p$ field $k$ to a formal group over the Witt vectors $W(k)$ is unobstructed \cite[Theorem 4.46]{MR767090}. Over such fields $k$, $p$-divisible formal groups can be studied via their Newton polygon and for abelian schemes these satisfy a restrictive symmetry condition. Hence the existence of lifts illustrates that our results apply to more than formal groups of abelian schemes, so that we have indeed uncovered new, genuinely $p$-adic, Manin--Mumford type results.\par
\begin{remark}
Lifting morphisms of formal groups on the other hand is more delicate. For instance over $\overline{\mathbb{F}}_p$, a $p$-divisible formal group is isogenous to a product of isoclinic groups corresponding to each slope of its Newton polygon, but the decomposition up to isogeny cannot be lifted in general. 
\end{remark}
This serves to highlight another hindrance in the formal setting: for abelian varieties, one could alternatively use Poincar\'e reducibility to split off isogeny complements of abelian subvarieties and then similarly induct on dimensions in the infinite stabilizer case. However, in the absence of suitable polarizations, even after base-changing to an extension of $R$  one may not have an isogeny complement for a formal group. For instance over $\calO_{{\IC}_p}$, from the Scholze-Weinstein classification \cite[Thm B]{ScholzeWeinstein} one can identify $p$-divisible groups $G$ up to isogeny with finite dimensional $\IQ_p$-vector spaces (by considering Tate modules $T_p(G)\otimes \IQ_p$), together with a filtration after tensoring with $\mathbb{C}_p$; this does not form a semisimple category. It would be interesting to see if these somewhat technical obstacles can be surmounted or if a new strategy can shed light on Question \ref{strongconjectureformal}.

\section{Rigid analytic multiplicative Manin--Mumford}\label{sec:rigid}

We now seek to establish similar results for $p$-adic analytic functions when the special points consist of all of the torsion points of a commutative algebraic group, not just the $p$-primary part. To that end, we consider power series converging on the ``closed'' $p$-adic unit polydisk $\calO_{\IC_p}^n$, namely we work with functions in the Tate algebra $T_n(K):=\{f=\sum_Ja_JX^J\in K[[X_1,\ldots, X_n]]\vert  a_J\to 0 \}$ for multi-indices $J\in\IN^n$ and look for a $p$-adic strengthening of the classical results of a rigid analytic flavor.\\
The so-called \emph{multiplicative} case of the Manin--Mumford Conjecture was already considered by S. Lang in \cite{MR0190146}, and states that an irreducible subvariety $V\hookrightarrow \IG_m^n(\IC)$ of the multiplicative group contains a Zariski-dense set of torsion points if and only if $V$ is a translate by a torsion point of a subtorus. When $n=2$, this amounts to the following result about polynomial functions, with proofs due to Y. Ihara, J.-P. Serre and J. Tate:

\begin{lemma}[\cite{MR0190146}, p. 230]
Let $C$ be an absolutely irreducible plane curve given by the zero set of a polynomial $f(X,Y)=0$. If $C$ passes through the multiplicative identity and if for infinitely many roots of unity $(\zeta,\xi)$ we have
$$f(\zeta, \xi)=0,$$
then up to a constant $f(X,Y)=X^m-Y^l$ or $f(X,Y)=X^mY^l-1$ for a pair of nonnegative integers $(m,l)\neq (0,0)$. 
\end{lemma}
We essentially consider the same questions replacing polynomials $f$ by analytic functions in $T_n(K)$. In particular, we obtain the following generalization:
\begin{prop}\label{prop:rigidLang}
Consider a one-dimensional affinoid $K$-algebra $T_2(K)/(f)$ for $f$ some irreducible two-variable power series in $T_2(K)$ passing through the identity. If for infinitely many roots of unity $(\zeta,\xi)$ we have 
$$f(\zeta,\xi)=0,$$
then, up to a unit in $T_2(K)$, $f$ must be a polynomial $f(X,Y)=X^m-Y^l$ or $f(X,Y)=X^mY^l-1$ for a pair of nonnegative integers $(m,l)\neq (0,0)$. 
\end{prop}
Similarly to the formal setting, we utilize the notion of stabilizer and distinguish between finite and infinite stabilizers.
\subsection{Finite Stabilizer}
We can deal with the $p$-primary torsion as in the formal setting and we shall use this approach to reduce to when unramified torsion points, which have order prime to $p$, are dense in the relevant affinoid space. One then has the following:
\begin{lemma}\label{lemma:unramifiedclosure} Let $X:=\Sp(T_n(K)/I)$ denote an affinoid space over $K$ for $K/\mathbb{Q}_p$ finite. If unramified torsion points of $\IG_m^n(\IC_p)$ are dense in $X$ (in the Zariski topology on the rigid unit ball), then $X$ is stable under taking $p$-th powers. 
\end{lemma}
\begin{proof}
Since the roots of unity are unramified, they are Teichm\"uller lifts from $\overline{\IF}_p$. Let $\tau$ denote a lift of the $p$-power Frobenius to $\Gal(K^{ur}/K)$ and denote $[p](X_1,\ldots,X_n)=(X_1^p,\ldots,X_n^p)$. Then for all $\zeta\in K^{ur}$ we have the identity $\tau(\zeta)-[p](\zeta)=0$. Letting $\Sigma$ denote the set of dense unramified torsion points, one obtains that $[p]X=[p]\overline{\Sigma}=\overline{[p]\Sigma}=\overline{\tau(\Sigma)}$. Since $\tau$ acts by continuous automorphisms, $\overline{\tau(\Sigma)}=\tau(X)$ and this yields $[p]X=\tau(X)=X$,
 proving the claim.

\end{proof}

\begin{remark}Although we limit ourselves to the multiplicative case, a version of Lemma \ref{lemma:unramifiedclosure} should hold more generally for coordinates of unramified torsion points of algebraic groups (such as abelian varieties) satisfying a polynomial identity in a lift of Frobenius, see for instance \cite[Lemma 3.2.]{MR1989204} and \cite[Lemma 3.4.]{MR2185637}. 
\end{remark}

\begin{lemma}\label{lemma:stable}
Let $X:=\Sp(T_n(K)/I)$ denote an affinoid space over $K$. If $X(\IC_p)$ is infinite and stable under taking $p$-th powers, then $X$ contains a translate by a torsion point of a one-dimensional (algebraic) subtorus of $\IG_m^n$. 
\end{lemma}
We note that over a number field subvarieties of a semi-abelian variety which are stable under the multiplication-by-$m$ map for some $m\in\IN$ and pass through the origin are known to be subgroups (see e.g., \cite{MR0190146,MR2176757}). Thus such a result is to be expected. 
\begin{proof}
We assume without loss of generality that the multiplicative identity is in $X(\IC_p)$. We first claim that there is a point of infinite order on $X(\IC_p)$ lying on the interior of the unit polydisk centered at the identity $1+\im_{\IC_p}^n$. Indeed, by Noether normalization, there exists a finite injective map 
$$T_d(K)\hookrightarrow T_n(K)/I$$
for some integer $d\geq 1$ that constitutes the Krull dimension of the affinoid algebra $T_n(K)/I$. Since $X$ passes through the multiplicative identity, we may find $x=(x_1,\ldots,x_n)\in X(\IC_p)$ such that 
\begin{enumerate}
   \item the valuation satisfies $v_p(x_i-1)>1/(p-1)$ for all $1\leq i\leq n$,
  \item for at least $d$ coordinates we have $x_i\neq 1$.
\end{enumerate}
Since the coordinates of torsion points have valuations $v_p(\zeta-1)\geq 1/(p-1)$ (throughout we normalize valuations so that $v_p(p)=1$), this establishes the claim. Taking the $p$-adic logarithm in each coordinate, which gives a group homomorphism $\log_{\IG_m^n}$ to $\IG_a^n(\IC_p)$, we find that the image of the points on $X(\IC_p)$ that lie inside the ball of radius $p^{-1/(p-1)}$ centered at the multiplicative identity has infinite intersection $\{ \log_{\IG_m^n}(x^{p^t})\vert t>1\}$ with a straight line and has zero as point of accumulation. Exponentiating, we find that $X(\IC_p)$ has infinite intersection with a one-dimensional algebraic subgroup $H$ of $\IG_m^n(\IC_p)$ and by Zariski density must contain $H$. 
\end{proof}

\subsection{Infinite Stabilizer}
We now consider the case when the stabilizer under multiplication of our space $X=\Sp(A)$ associated to an affinoid $K$-algebra $A:=T_n(K)/I$ is infinite and show that $X$ has to contain a positive-dimensional algebraic subgroup. 
\begin{lemma}\label{lemma:rigidstab}
Assume that the stabilizer $\Stab(X)$ of an affinoid space $X$ passing through the identity contains infinitely many algebraic points. Then $X$ contains a positive dimensional algebraic subtorus. 
\end{lemma}
\begin{proof}
We proceed similarly to Lemma \ref{lemma:infinitestab}. Since affinoid algebras $T_n(\overline{K})/I$ are Noetherian, the intersection on the right hand side of 
$$\Stab(X)(\overline{K})=\bigcap_{Q \in X(\overline{K})}T_QX(\overline{K})$$
is finite, denoting similarly to the formal case $T_QX(\overline{K})=\{Q^{-1}\cdot x\vert x\in X(\overline{K})\}$. Thus $\Stab(X)$ is a positive dimensional affinoid subspace over a finite extension $L$ of $K$ and is contained in $X$. Moreover, $\Stab(X)(\IC_p)$ is a subgroup of the multiplicative group $\IG_m^n(\IC_p)$ and therefore the irreducible component of $\Stab(X)$ passing through the origin is positive dimensional and has dense intersection with an algebraic subtorus. The result follows.   
\end{proof}

Alternatively, we sketch an argument using the approach in A. Neira's thesis \cite{MR2705372}. We proceed by induction on the number $n$ of variables. \\

Let $f\in T_n(K)$ which without loss of generality may be assumed to have $\calO_K$-coefficients.
We also assume $f$ is $X_i$-distinguished for some $1\leq i\leq n$. We may as well assume $i=n$ and apply Weierestrass preparation (see \cite[Section 5.2.2]{MR746961}). Since units do not vanish on the $p$-adic polydisk, we may assume that $f$ is a polynomial in $X_n$ of degree $\nu$. Let $\rho \in T_{n-1}(K)[X_n]$ cut out an irreducible component $W$ of $\Sp(T_n(K)/(f))$. 

\begin{lemma}[\cite{MR2705372}]\label{lemma:relation} For any $(\alpha_1,\ldots, \alpha_n)\in \Stab(W)$ we must have that 
\begin{equation}\label{eq1}\rho(\alpha_1\cdot X_1,\ldots, \alpha_n\cdot X_n)=\alpha_n^\nu \cdot \rho (X_1,\ldots, X_n).
\end{equation}
\end{lemma}
\begin{proof}
This is proved in \cite[Chapter 4]{MR2705372}. We briefly sketch the idea: since $\alpha$ stabilizes irreducible $W$ we must have that 
\begin{equation}\label{eq2}\rho(\alpha_1\cdot X_1,\ldots, \alpha_n\cdot X_n)=\lambda(X_1,\ldots, X_n)\cdot \rho (X_1,\ldots, X_n)\end{equation}
for some unit $\lambda$. Because the power series involved are polynomials in $X_n$, one can show that $\lambda(X)$ is a power series in just the variables $(X_1,\ldots, X_{n-1})$. The identity then follows by comparing coefficients of the polynomials in $X_n$ of equation \eqref{eq2}.
\end{proof}

The group structure on points of the stabilizer yields a homomorphism
\begin{align*} \IZ^n&\rightarrow \Hom_{\Grp} (\Stab(W)(\calO_{\IC_p})\to \calO_{\IC_p}^\times)\\
\lambda &\to (\alpha_1,\ldots, \alpha_n)\mapsto \alpha_1^{\lambda_1}\cdots\alpha_n^{\lambda_n},
\end{align*}
and we denote by $\Lambda$ the kernel of this map. One easily obtains the following relation:
\begin{lemma}[\cite{MR2705372}]\label{lemma:Lambda} If the stabilizer is infinite, then $\Lambda$ is contained in a linear hyperplane given by $\sum_{i=1}^n c_i\lambda_i=0$ for some integers $c_i$. 
\end{lemma}
\begin{proof}
Assume not, so that $\Lambda$ contains $n$ linearly independent vectors and therefore contains $(0,\ldots,\lambda_j,\ldots 0)$ for any $1\leq j\leq n$ and some appropriate $\lambda_j$. It follows that $\alpha_j^{\lambda_j}=1$ for $1\leq j\leq n$ if $\alpha\in \Stab(W)$.
\end{proof}
Thus for $c_i$ as in Lemma \ref{lemma:Lambda}, assuming that $c_n\neq 0$ and that $\eta$ is a root of unity chosen such that $\zeta_n=\eta^{-c_n}$, we may write:
\begin{align*}
\rho(\zeta)&=\sum \rho_{i_1,\ldots,i_n}\zeta_1^{i_1}\cdots \zeta_n^{i_n}\\
&=\sum\rho_{i_1,\ldots,i_n}(\zeta_1\eta^{c_1})^{i_1}\cdots (\zeta_{n-1}\eta^{c_{n-1}})^{i_{n-1}} \cdot\zeta_n^{i_n}\eta^{-\sum_{k=1}^{n-1} c_ki_k}\\
&=\sum\rho_{i_1,\ldots,i_n}(\zeta_1\eta^{c_1})^{i_1}\cdots (\zeta_{n-1}\eta^{c_{n-1}})^{i_{n-1}} \cdot\zeta_n^{i_n}\eta^{c_n(i_n-\nu)}\\
&=\zeta_n^{\nu}\sum\rho_{i_1,\ldots,i_n}(\zeta_1\eta^{c_1})^{i_1}\cdots (\zeta_{n-1}\eta^{c_{n-1}})^{i_{n-1}},
\end{align*}
using that $\rho_{i_1,\ldots,i_n}\neq 0$ implies $(i_1,\ldots,i_{n-1}, i_n-\nu)\in \Lambda$ by equating monomials in $X_n$ in equation \eqref{eq1} of Lemma \ref{lemma:relation}. 
We deduce therefore that $\overline{\rho}(X):=\rho(X_1,\ldots,X_{n-1},1)$ vanishes at the roots of unity $(\zeta_1\eta^{c_1},\ldots, \zeta_{n-1}\eta^{c_{n-1}})$ for $\zeta$ a torsion point on $W$. One may now consider $\overline{\rho}\in T_{n-1}(K)$ and proceed by induction. Similarly one may induct when $c_n=0$. In this way Neira shows that for such spaces $\Sp(T_n(K)/(f))$ the analogue of the Tate--Voloch Conjecture holds in this setting, namely there exists a positive constant $C$ such that for all torsion points $\zeta\in\IG_m^n(\IC_p)$ either $f(\zeta)=0$ or $\vert f(\zeta)\vert_p>C$, see \cite[Theorem 4.1.]{MR2705372}. We may leverage this method to prove the Manin--Mumford type statement for $W$ as well: consider the group homomorphism
\begin{align*}
\pi_\eta:\IG_m^{n, tor}(\IC_p)&\to \IG_m^{n-1, tor}(\IC_p)\\
(\zeta_1,\ldots,\zeta_n)&\mapsto (\zeta_1\eta^{c_1},\ldots, \zeta_{n-1}\eta^{c_{n-1}}).
\end{align*}
Since $\overline{\rho}(X):=\rho(X_1,\ldots,X_{n-1},1)$ vanishes along $\pi_\eta(W\cap\IG_m^{n, tor}(\IC_p) )$, one sees that if $\rho$ vanishes along the torsion on an algebraic subtorus, so does $\overline{\rho}$. One concludes by induction on $n$, the case of $n=1$ following by Weierstrass preparation.
\par 
 
In this way one may deduce the result for some spaces of the form $\Sp(T_n(K)/(f))$. However, the drawback is that one needs the power series involved in the argument to be distinguished in a variable. For a general power series, this may always be arranged after twisting by a $K$-algebra automorphism $\sigma$ of $T_n(K)$ given by $\sigma(X_i)=X_i+X_n^{c_i}$ for $c_i\in\IN$, $1\leq i<n$ and $\sigma(X_n)=X_n$ for suitable choices of $c_i$ (see, e.g., \cite[Section 5.2.4]{MR746961}). But these $K$-algebra automorphisms may not preserve the group structure we are considering and in particular the torsion points, so an additional argument is needed to obtain results for general $f\in T_n(K)$ with this approach.  
\par

Returning to our previous approach, using Lemma \ref{lemma:rigidstab} we are able to deduce a general result:
\begin{theorem}\label{thm:rigid}
Let $X=\Sp(A)$ denote the rigid space associated to an affinoid $K$-algebra $A:=T_n(K)/I$ for some finite extension $K/\IQ_p$. If $X(\IC_p)$ contains infinitely many torsion points of $\IG_m^n(\IC_p)$, then $X$ contains a positive-dimensional algebraic torus. Moreover, the same conclusion holds if infinitely many torsion points approach $X(\IC_p)$ for a suitable $p$-adic distance. 
\end{theorem}
\begin{proof}
We proceed again by induction on $n$, the case when $n=1$ following by Weierstrass preparation. Assume the stabilizer of $X$ has infinitely many algebraic points. Then the result holds by Lemma \ref{lemma:rigidstab}. \\
We may therefore suppose the stabilizer of $X$ has finitely many algebraic points. As in the formal case, the action of inertia via Lemma \ref{Boxallformal} implies that given a torsion point $\zeta=\zeta'\cdot\zeta_p$, with $\zeta_p$ denoting the projection to $p$-power torsion of $\zeta$ and $\zeta_p\not\in \IG_m^n(K(\mu_{p^r}))$ for $r\geq 1$, there is $\sigma\in \Gal(\overline{K}/K^{ur})$ such that 
$$\sigma(\zeta)/\zeta=\sigma(\zeta_p)/\zeta_p\in \mu_{p^r}^n\setminus \mu_{p^{r-1}}^n.$$
The analogue of Lemma \ref{translateclose} in this setting now implies that the torsion $\varepsilon$-close to $X$ satisfies
$$X(\varepsilon)\subseteq \{\zeta\vert \zeta_p\in \mu_{p^{r-1}}^n\}\cup\bigcup_{Q\in \mu_{p^{r}}^n\setminus \mu_{p^{r-1}}^n} (X\cap T_{Q}X)(\varepsilon).$$
Let $r$ be chosen large enough so that for any $\zeta=\zeta'\cdot \zeta_p\in\Stab(X)$ we have that $\zeta_p\in \mu_{p^{r-1}}$. Assuming $X$ is irreducible, we obtain that $X\cap T_{Q}X$ is lower-dimensional for $Q\in \mu_{p^{r}}^n\setminus \mu_{p^{r-1}}^n$ and handle this case by induction. We may therefore assume the $p$-power part of the torsion in $X(\varepsilon)$ is finite.
After working with coefficients over a finite ramified extension of $K$ it therefore suffices to consider infinite unramified torsion in $X(\varepsilon)$. For unramified torsion, by \cite[Theorem 2.1.]{MR2705372} it is known that for $\varepsilon>0$ small enough $\zeta\in X(\varepsilon)$ implies $\zeta\in X$. Finally, without loss of generality we may assume that the unramified torsion points are dense in $X$ for $\varepsilon>0$ small enough. The result now follows from Lemmas \ref{lemma:unramifiedclosure} and \ref{lemma:stable}.

\end{proof}

In particular this implies Proposition \ref{prop:rigidLang} when $n=2$. Moreover we obtain for general $n$ a rigid analytic version of the multiplicative Manin--Mumford result as well as of the Tate--Voloch Conjecture for one-dimensional affinoid spaces. Similarly to the formal setting, we suspect this result can be strengthened to show that if the closure of the torsion points on an affinoid $X$ is $d$-dimensional, $X$ contains a $d$-dimensional algebraic torus. 
\section{Acknowledgements}
We would like to thank Laurent Berger for helpful conversations on $p$-adic dynamical systems. We would also like to thank Sean Howe for his availability to discuss $p$-divisible groups and for his valuable feedback. Finally, we thank Vincent Pilloni, Jacob Tsimerman and Felipe Voloch for beneficial discussions on topics relevant to this paper.

\nocite{}
\bibliography{rigidbiblio.bib}

\begin{thebibliography}{{Ber}20}

\bibitem[BD11]{MR2817647}
Matthew Baker and Laura DeMarco.
\newblock Preperiodic points and unlikely intersections.
\newblock {\em Duke Math. J.}, 159(1):1--29, 2011.

\bibitem[{Ber}20]{2018arXiv181105824B}
Laurent {Berger}.
\newblock {Rigidity and unlikely intersections for formal groups}.
\newblock {\em Acta Arithmetica (published online)}, May 2020.

\bibitem[BGR84]{MR746961}
S.~Bosch, U.~G\"{u}ntzer, and R.~Remmert.
\newblock {\em Non-{A}rchimedean analysis}, volume 261 of {\em Grundlehren der
  Mathematischen Wissenschaften [Fundamental Principles of Mathematical
  Sciences]}.
\newblock Springer-Verlag, Berlin, 1984.
\newblock A systematic approach to rigid analytic geometry.

\bibitem[BGT16]{MR3468757}
Jason~P. Bell, Dragos Ghioca, and Thomas~J. Tucker.
\newblock {\em The dynamical {M}ordell-{L}ang conjecture}, volume 210 of {\em
  Mathematical Surveys and Monographs}.
\newblock American Mathematical Society, Providence, RI, 2016.

\bibitem[Box95]{MR1345173}
John Boxall.
\newblock Sous-vari{\'e}t{\'e}s alg{\'e}briques de vari{\'e}t{\'e}s
  semi-ab{\'e}liennes sur un corps fini.
\newblock In {\em Number theory ({P}aris, 1992--1993)}, volume 215 of {\em
  London Math. Soc. Lecture Note Ser.}, pages 69--80. Cambridge Univ. Press,
  Cambridge, 1995.

\bibitem[Bre00]{MR1804530}
Christophe Breuil.
\newblock Groupes {$p$}-divisibles, groupes finis et modules filtr\'{e}s.
\newblock {\em Ann. of Math. (2)}, 152(2):489--549, 2000.

\bibitem[GY18]{MR3801489}
Dragos Ghioca and Hexi Ye.
\newblock A dynamical variant of the {A}ndr\'{e}-{O}ort conjecture.
\newblock {\em Int. Math. Res. Not. IMRN}, (8):2447--2480, 2018.

\bibitem[Lan65]{MR0190146}
Serge Lang.
\newblock Division points on curves.
\newblock {\em Ann. Mat. Pura Appl. (4)}, 70:229--234, 1965.

\bibitem[Lub94]{MR1310863}
Jonathan Lubin.
\newblock Non-{A}rchimedean dynamical systems.
\newblock {\em Compositio Math.}, 94(3):321--346, 1994.

\bibitem[MS14]{MR3126567}
Alice Medvedev and Thomas Scanlon.
\newblock Invariant varieties for polynomial dynamical systems.
\newblock {\em Ann. of Math. (2)}, 179(1):81--177, 2014.

\bibitem[Nei02]{MR2705372}
Ana Raissa~Berardo Neira.
\newblock {\em Power series in p-adic roots of unity}.
\newblock ProQuest LLC, Ann Arbor, MI, 2002.
\newblock Thesis (Ph.D.)--The University of Texas at Austin.

\bibitem[PR02]{MR1989204}
Richard Pink and Damian Roessler.
\newblock On {H}rushovski's proof of the {M}anin-{M}umford conjecture.
\newblock In {\em Proceedings of the {I}nternational {C}ongress of
  {M}athematicians, {V}ol. {I} ({B}eijing, 2002)}, pages 539--546. Higher Ed.
  Press, Beijing, 2002.

\bibitem[Ray83a]{MR688265}
M.~Raynaud.
\newblock Courbes sur une vari\'{e}t\'{e} ab\'{e}lienne et points de torsion.
\newblock {\em Invent. Math.}, 71(1):207--233, 1983.

\bibitem[Ray83b]{MR717600}
M.~Raynaud.
\newblock Sous-vari\'{e}t\'{e}s d'une vari\'{e}t\'{e} ab\'{e}lienne et points
  de torsion.
\newblock In {\em Arithmetic and geometry, {V}ol. {I}}, volume~35 of {\em
  Progr. Math.}, pages 327--352. Birkh\"{a}user Boston, Boston, MA, 1983.

\bibitem[Roe05]{MR2176757}
Damian Roessler.
\newblock A note on the {M}anin-{M}umford conjecture.
\newblock In {\em Number fields and function fields---two parallel worlds},
  volume 239 of {\em Progr. Math.}, pages 311--318. Birkh\"{a}user Boston,
  Boston, MA, 2005.

\bibitem[Sca00a]{TV2Scanlon}
Thomas Scanlon.
\newblock The conjecture of tate and voloch on p-adic proximity to torsion.
\newblock {\em Int. Math. Res. Not. IMRN}, 17, 07 2000.

\bibitem[Sca00b]{TV1Scanlon}
Thomas Scanlon.
\newblock p-adic distance from torsion points of semi-abelian varieties.
\newblock {\em Journal fur die Reine und Angewandte Mathematik}, 499, 07 2000.

\bibitem[Sca05]{MR2185637}
Thomas Scanlon.
\newblock A positive characteristic {M}anin-{M}umford theorem.
\newblock {\em Compos. Math.}, 141(6):1351--1364, 2005.

\bibitem[Ser18]{Serban:2016aa}
Vlad Serban.
\newblock An infinitesimal $p$-adic multiplicative manin--mumford conjecture.
\newblock {\em Journal de Th\'eorie des Nombres de Bordeaux}, 30(2):393--408,
  2018.

\bibitem[SW13]{ScholzeWeinstein}
Peter Scholze and Jared Weinstein.
\newblock Moduli of $p$-divisible groups.
\newblock {\em Cambridge Journal of Mathematics}, 1:145--237, 01 2013.

\bibitem[TV96]{MR1405976}
John Tate and Jos{{\'e}}~Felipe Voloch.
\newblock Linear forms in {$p$}-adic roots of unity.
\newblock {\em Internat. Math. Res. Notices}, (12):589--601, 1996.

\bibitem[Vol96]{MR1426536}
Jos{\'e}~Felipe Voloch.
\newblock Integrality of torsion points on abelian varieties over {$p$}-adic
  fields.
\newblock {\em Math. Res. Lett.}, 3(6):787--791, 1996.

\bibitem[Zan12]{MR2918151}
Umberto Zannier.
\newblock {\em Some problems of unlikely intersections in arithmetic and
  geometry}, volume 181 of {\em Annals of Mathematics Studies}.
\newblock Princeton University Press, Princeton, NJ, 2012.
\newblock With appendixes by David Masser.

\bibitem[Zha06]{MR2408228}
Shou-Wu Zhang.
\newblock Distributions in algebraic dynamics.
\newblock In {\em Surveys in differential geometry. {V}ol. {X}}, volume~10 of
  {\em Surv. Differ. Geom.}, pages 381--430. Int. Press, Somerville, MA, 2006.

\bibitem[Zin84]{MR767090}
Thomas Zink.
\newblock {\em Cartiertheorie kommutativer formaler {G}ruppen}, volume~68 of
  {\em Teubner-Texte zur Mathematik [Teubner Texts in Mathematics]}.
\newblock BSB B. G. Teubner Verlagsgesellschaft, Leipzig, 1984.
\newblock With English, French and Russian summaries.

\end{thebibliography}

\end{document}